\documentclass[10pt]{amsart}
\usepackage{amssymb}
\usepackage[T1]{fontenc}
\usepackage{dsfont}
\usepackage{amscd}
\usepackage{mathabx}
\usepackage[mathscr]{euscript} % for the power-set P
\usepackage[all]{xy}
\usepackage{url}
\usepackage{xcolor}
\usepackage{stmaryrd}
\usepackage{comment}
\usepackage[retainorgcmds]{IEEEtrantools}
\usepackage{hyperref}
\usepackage{amsbsy}
\usepackage{enumerate,xspace}
\usepackage[all]{xy}
\usepackage{tikz-cd}
\title[Model theory of Heisenberg group]{Some model theory of the Heisenberg group}

\author[M. Fr\c{a}cek]{Maciej Fr\k{a}cek$^{\clubsuit}$}

\address{$^{\clubsuit}$Instytut Matematyczny\\
Uniwersytet Wroc{\l}awski\\
Wroc{\l}aw\\
Poland}
\email{339724@uwr.edu.pl}

\author[P. KOWALSKI]{Piotr Kowalski$^{\diamondsuit}$}
\thanks{$^{\diamondsuit}$
 Supported by the Narodowe Centrum Nauki grant no. 2021/43/B/ST1/00405 and by the T\"{u}bitak grant no. 1001-124F359.}
\address{$^{\diamondsuit}$Instytut Matematyczny\\
Uniwersytet Wroc{\l}awski\\
Wroc{\l}aw\\
Poland}
\email{pkowa@math.uni.wroc.pl} \urladdr{http://www.math.uni.wroc.pl/\textasciitilde pkowa/ }

\DeclareMathOperator{\id}{id}

\DeclareMathOperator{\alg}{alg}

\DeclareMathOperator{\Th}{Th}

\makeatletter
\def\moverlay{\mathpalette\mov@rlay}
\def\mov@rlay#1#2{\leavevmode\vtop{%
   \baselineskip\z@skip \lineskiplimit-\maxdimen
   \ialign{\hfil$\m@th#1##$\hfil\cr#2\crcr}}}
\newcommand{\charfusion}[3][\mathord]{
    #1{\ifx#1\mathop\vphantom{#2}\fi
        \mathpalette\mov@rlay{#2\cr#3}
      }
    \ifx#1\mathop\expandafter\displaylimits\fi}
\makeatother

\def\Ind#1#2{#1\setbox0=\hbox{$#1x$}\kern\wd0\hbox to 0pt{\hss$#1\mid$\hss}
\lower.9\ht0\hbox to 0pt{\hss$#1\smile$\hss}\kern\wd0}

\def\Notind#1#2{#1\setbox0=\hbox{$#1x$}\kern\wd0\hbox to 0pt{\mathchardef
\nn=12854\hss$#1\nn$\kern1.4\wd0\hss}\hbox to
0pt{\hss$#1\mid$\hss}\lower.9\ht0 \hbox to
0pt{\hss$#1\smile$\hss}\kern\wd0}

\newtheorem{theorem}{Theorem}[section]
\newtheorem*{theorem*}{Main Theorem}

\newtheorem{prop}[theorem]{Proposition}
\newtheorem{lemma}[theorem]{Lemma}

\newtheorem{fact}[theorem]{Fact}
\theoremstyle{definition}
\newtheorem{definition}[theorem]{Definition}

\newtheorem{remark}[theorem]{Remark}

\begin{document} %rozpoczecie dokumentu%

\begin{abstract}
    We show that a field $K$ is model complete (in the language of rings) if and only if the Heisenberg group $H(K)$ is model complete (in the language of groups). To show that, we extend Levchuk's result about automorphisms of $H(K)$ to the case of monomorphisms $H(K)\to H(M)$. We also show that $H(K)$ does not have quantifier elimination and discuss its (non-)bi-interpretability with $K$.
\end{abstract}

\maketitle

\newcommand{\Cc}{{\mathds{C}}}
\newcommand{\Rr}{{\mathds{R}}}
\newcommand{\Nn}{{\mathds{N}}}
\newcommand{\Qq}{{\mathds{Q}}}
\newcommand{\Zz}{{\mathds{Z}}}
\newcommand{\Pp}{{\mathds{P}}}
\newcommand{\Aa}{{\mathds{A}}}
\newcommand{\Ff}{{\mathds{F}}}
\newcommand{\Uu}{{\mathds{U}}}

\section{Introduction}

In this paper, we study model completeness of groups of rational points of the Heisenberg group. Model completeness is a weaker variant of quantifier elimination, where the formulas can be reduced to ones having only existential quantifiers. There are many classical structures which are model complete but do not enjoy quantifier elimination. Examples include the field $\mathbb{R}$ of real numbers~\cite[Theorem 2.7.3]{HoMo}, the field $\mathbb{Q}_p$ of $p$-adic numbers~\cite{Mac76}, perfect PAC fields satisfying some additional Galois-theoretic conditions~\cite{mcefree}, and the exponential field~$(\mathbb{R}, \text{exp})$ of real numbers~\cite{Wilkie}.

If $K$ is a field, then the Heisenberg group $H(K)$ is the group of upper unitriangular $3$ by $3$ matrices with coefficients from $K$. The main result of this paper is as follows.

\begin{theorem*}
Let $K$ be a field. Then the group $H(K)$ is model complete if and only if $K$ is a model complete field.
\end{theorem*}

In general, the model completeness of an algebraic group is not guaranteed, even if the underlying field is model complete. For example, the multiplicative group $\mathbb{Q}_p^\times$ is not model complete for $p > 2$, since the map
$$\mathbb{Q}_p^\times\ni x\mapsto x^p\in \mathbb{Q}_p^\times$$
is a monomorphisms which is not onto. The converse statement also does not hold. For instance, the group $(\mathbb{F}_p(X), +)$ is model complete even though the field $\mathbb{F}_p(X)$ is not.

Similar results were recently obtained in \cite{HKTY} in the case of semisimple split algebraic groups replacing the Heisenberg group. Clearly, our results in this note should generalize to all unitriangular matrix groups and possibly to many other types of unipotent algebraic groups. However, we prefer to have a clear account regarding the basic case of the Heisenberg group first, hence we leave further generalizations to subsequent work.

Model theory of the Heisenberg group has been extensively studied starting from Maltsev's interpretation of $K$ in $H(K)$ (using two extra constants) in \cite{Maltsev}. The authors of \cite{Int} considered an interpretation of $K$ in $H(K)$ without any extra constants. However, the questions of model completeness have not been addressed yet. There is also work of Belegradek \cite{bel94} on model theory of general unitriangular groups over arbitrary rings. Finally, it is shown in \cite{khelif} and \cite{niesgps} that the ring $\Zz$ is not bi-interpretable with the group $H(\Zz)$, see also a very recent work of Danyarova and Myasnikov \cite{danmya} on general theory of (bi-)interpretations.

This paper is organised as follows. In Section \ref{secmon}, we analyse monomorphisms between the rational points of the Heisenberg group and put them into a more general context of central group extensions. The proof of the main result of Section \ref{secmon} follows the steps of the argument of Levchuk from \cite{levchuk83}, where the case of automorphisms was considered. In Section \ref{secmc}, we recall first some basic definitions and facts from model theory regarding interpretability and model completeness. Then, we recall the Maltsev's interpretation of a field in its Heisenberg group from \cite{Maltsev} and afterwards prove the main theorem. In Section \ref{seclast}, we show (using results of Cherlin and Felgner from \cite{cherlin82}) that $H(K)$ does not have quantifier elimination and discuss some questions from Section 5 of \cite{Int} concerning the model theory of Heisenberg groups.

The research from this paper originates from the Bachelor Thesis of the first author which was written under the supervision of the second author.

We would like to thank the members of the model theory group in Wroc{\l}aw for their constructive remarks during the talk of the first author at the model theory seminar at Wroc{\l}aw University.

\section{Monomorphisms}\label{secmon}
In this section, we analyse monomorphisms between Heisenberg groups (over different fields). The material here mostly comes from  \cite{gibbs70} and \cite{levchuk83}, however we simplify and clarify a little the account from \cite{levchuk83} regarding the Heisenberg group by putting it in a more general context of central group extensions and (more importantly) generalize some of the results appearing in \cite{levchuk83} from automorphisms of $H(K)$ to monomorphisms $H(K)\to H(M)$.

\subsection{Automorphisms of central group extensions}\label{secgen}
We start with some general observations about automorphisms of central group extensions. This material is a special case of the results from  \cite{wellsc}, where arbitrary group extensions were analysed.

Let us consider the following short exact sequence of groups
$$1\longrightarrow B\longrightarrow G\longrightarrow A\longrightarrow 1,$$
where $(A,+)$ and $(B,+)$ are commutative and $B$ is mapped into $Z(G)$. Then, there is a cocycle $c\in Z^2(A,B)$ such that $G$ is isomorphic to a group with the universe $A\times B$ and with the following group operation
$$(a,b)\cdot (a',b')=(a+a',b+b'+c(a,a')).$$
We are interested in automorphisms of $G$. Let us fix $\alpha\in \mathrm{Aut}(A)$, $\beta\in \mathrm{Aut}(B)$ and $\gamma:A\to B$. We define the following map
$$\Psi:G\to G,\quad \Psi(a,b)=(\alpha(a),\beta(b)+\gamma(a)).$$
In such a way, we get an action of the group of 2 by 2 ``lower triangular matrices'' having $\alpha,\beta$ on the diagonal and $\gamma$ down the diagonal. In particular, all such maps $\Psi$ are bijections.

\begin{remark}
We were a little puzzled about this group of ``matrices'' above (see also the structure appearing at the formulation of Lemma in \cite[Section 3]{wellsc}). The only general explanation we could figure out is coming from bi-actions by group automorphisms. Suppose than $L,R,N$ are groups, $L$ acts on $N$ by automorphisms on the left, $R$ acts on $N$ by automorphisms on the right and for all $l\in L,r\in R$ and $n\in N$ we have
$$l\cdot (n\cdot r)=(l\cdot n)\cdot r.$$
Then, there is a group structure on the set $L\times N\times R$ given by
$$(l,n,r)*(l',n',r'):=(ll',(l\cdot n')(n\cdot r'),rr').$$
In our case, we have
$$L:=\mathrm{Aut}(A),\quad R:=\mathrm{Aut}(B),\quad N:=\mathrm{Functions}(A,B)$$
and the bi-action is given by the pre-composition and the post-composition of functions.
\end{remark}

By straightforward computations, we obtain the following, which is also a special case of Lemma in \cite[Section 3]{wellsc}.
\begin{lemma}\label{homcond}
The map $\Psi$ above is an automorphism of $G$ if and only if for all $a,a'\in A$, we have
$$\beta(c(a,a'))+\gamma(a+a')=\gamma(a)+\gamma(a')+c(\alpha(a),\alpha(a')).$$
\end{lemma}

\subsection{Monomorphisms between Heisenberg groups}\label{secheis}

Let $K$ be a field. The Heisenberg group
$$H(K)=\mathrm{UT}_3(K)$$
is the group of upper unitriangular $3$ by $3$ matrices, that is matrices of the form:
$$(a, b, c)_K:= \begin{pmatrix}
1 & a & c \\
0 & 1 & b \\
0 & 0 & 1
\end{pmatrix}$$
for $a,b,c \in K$. We will often just write ``$(a,b,c)$'' instead of ``$(a,b,c)_K$''. We regard $H$ as a functor from the category of fields to the category of groups.

We have the following short exact sequence of groups:
\[\begin{tikzcd}
	0 & {(K, +)} & {H(K)} & {(K^2, +)} & 0
	\arrow[from=1-1, to=1-2]
	\arrow[from=1-2, to=1-3]
	\arrow[from=1-3, to=1-4]
	\arrow[from=1-4, to=1-5]
\end{tikzcd}\]
where the second map takes $c \in K$ to $(0,0,c)\in H(K)$, and the third one maps $(a,b,c) \in H(K)$ to $(a,b)\in K^2$. Thus, $H(K)$ is an extension of $(K^2, +)$ by $(K, +)$. Furthermore, the image of the second map is precisely the center of $H(K)$, so this is a central extension. The center of $H(K)$ coincides with its commutator subgroup, and it is the subgroup of matrices with only the upper right corner possibly non zero, that is we have the following
$$Z\left(H(K)\right) = \left[H(K), H(K)\right] = \left\{(0,0,c) \mid c \in K \right\}.$$
Therefore, this situation fits perfectly to the set-up of Section \ref{secgen}. It is easy to check that the corresponding cocycle is
$$c\in Z^2\left((K,+)^2,(K,+)\right),\quad c((x,y),(x',y'))=xy'.$$
We also note the following formula for commutators in $H(K)$, which will be used in the sequel:
%$$(a, b, c)^{-1} = (-a, -b, -c + ab),$$
$$\left[ (a, b, c),(a', b', c') \right] = (0, 0, ab' - ba').$$

Following \cite{levchuk83}, we consider a special type of automorphisms of the Heisenberg group which we interpret using the terminology from  Section \ref{secgen}. It will turn out (as in \cite{levchuk83}) that there are no others. Let us fix
$$A=\begin{pmatrix}
 a & b \\
 c & d
\end{pmatrix}\in \mathrm{GL}_2(K).$$
We take as $\alpha$ the following $K$-linear map
$$\alpha:K^2\to K^2,\quad \alpha(x,y)=(ax+by,cx+dy)$$
and as $\beta$ a $K$-linear map as well
$$\beta:K\to K,\quad \beta(z)=\det(A)z.$$
Finally, we take $\gamma$ of the following form
$$\gamma:K^2\to K,\quad \gamma(x,y)=\Psi_1(x)+\Psi_2(y)+bcxy$$
for some $\Psi_1,\Psi_2:K\to K$. Using Lemma \ref{homcond}, we obtain the following.
\begin{lemma}\label{homheis}
The corresponding map
$$\Psi(x,y,z)=(A(x,y),\det(A)z+\Psi_1(x)+\Psi_2(y)+bcxy)$$
is an automorphism if and only if for all $x_1,x_2,y_1,y_2\in K$ we have
$$\Psi_1(x_1+x_2)+\Psi_2(y_1+y_2)=\Psi_1(x_1)+\Psi_1(x_2)+\Psi_2(y_1)+\Psi_2(y_2)+acx_1x_2+bdy_1y_2.$$
\end{lemma}
\begin{proof}
Let us take arbitrary
$$\bar{x}=(x_1,x_2)\in K^2,\quad \bar{y}=(y_1,y_2)\in K^2.$$
Then, using Lemma \ref{homcond}, $\Psi$ is an automorphism if and only if
\begin{equation}
\gamma(\bar{x}+\bar{y})-\gamma(\bar{x})-\gamma(\bar{y})=c(A\bar{x},A\bar{y})-\det(A)c(\bar{x},\bar{y}).\tag{$*$}
\end{equation}
Since $c(\bar{x},\bar{y})=x_1y_2$, the right-hand side in $(*)$ coincides with
$$acx_1y_1+bdx_2y_2+bc\left(x_1y_2+x_2y_1\right).$$
It is easy to check that for $\gamma(\bar{x})=\Psi_1(x_1)+\Psi_2(x_2)+bcx_1x_2$, the left-hand side in $(*)$ coincides with
$$\Psi_1(x_1+x_2)+\Psi_2(y_1+y_2)-\Psi_1(x_1)-\Psi_1(x_2)-\Psi_2(y_1)-\Psi_2(y_2)+bc\left(x_1y_2+x_2y_1\right),$$
which gives the result.
\end{proof}

Therefore, we fix $A$ as above and $\Psi_1,\Psi_2$ satisfying for all $x,x'\in K$
$$\Psi_1(x+x')=\Psi_1(x)+\Psi_1(x')+acxx',\quad \Psi_2(x+x')=\Psi_2(x)+\Psi_2(x')+bdxx',$$
which is exactly the set-up of Levchuk from \cite{levchuk83}. Levchuk shows then that \emph{all} automorphisms of $H(K)$ are of this form modulo the ones induced by field automorphisms of $K$. We will extend this result to monomorphisms of the form $H(K)\to H(M)$. To this end, we need the following result.
\begin{lemma}\label{psi_lemma}
Let $K\subseteq M$ be a field extension, $c \in K$, and $\psi : K \to K$ be
such that for all $x, y \in K$, we have
$$\psi(x + y) = \psi(x) + \psi(y) + cxy.$$
Then, we can extend $\psi$ to $\Psi : M \to  M$ such that for all $x, y \in M$, we have
$$\Psi(x + y) = \Psi(x) + \Psi(y) + cxy.$$
\end{lemma}
\begin{proof}
Let $L$ be a $K$-linear subspace of $M$ such that $M = K \oplus L$.
%Every element of $M$ is a sum of some elements from $K$ and $L$ in a unique way, therefore it is enough to define $\Psi$ on all elements of the %form $k + l$ for all $k \in K,\ l \in L$. \\
We notice that
$$\psi(0) = \psi(0 + 0) = \psi(0) + \psi(0) + 0.$$
Therefore, we obtain $\psi(0) = 0$.

If $\mathrm{char}(K) = 2$, then we have:
$$0 = \psi(0) = \psi(1 + 1) = \psi(1) + \psi(1) + c\cdot 1^2 = c.$$
Therefore, we get that $c = 0$. Thus, for all $k \in K$ and $l \in L$, we can simply define
$$\Psi(k + l):= \psi(k).$$
If $\mathrm{char}(K) \neq 2$, then for all $k \in K$ and $l \in L$ we set
$$\Psi(k + l):= \psi(k) + \frac{c}{2}l^2 + ckl.$$
We check below that $\Psi$ satisfies the desired equality. For any $k, k' \in K$ and $l,l' \in L$, we have the following
\begin{IEEEeqnarray*}{rCl}
\Psi((k + l) + (k' + l')) & = & \Psi((k + k') + (l + l')) \\
& = & \psi(k + k') + \frac{c}{2}(l + l')^2 + c(k + k')(l + l') \\
& = & \psi(k) + \psi(k') + ckk' + \frac{c}{2}l^2 + \frac{c}{2}l'^2 + cll' + ckl + ck'l' + ckl' + ck'l \\
& = & \Psi(k + l) + \Psi(k' + l') + ckk' + cll' + ckl' + clk' \\
& = & \Psi(k + l) + \Psi(k' + l') + c(k + l)(k' + l').
\end{IEEEeqnarray*}
Therefore, we see that $\Psi$ satisfies the necessary condition.
\end{proof}
We are ready now to show the main result of this section.
\begin{theorem}\label{homthm}
Let $K$ and $M$ be infinite fields, and $\Psi: H(K) \to H(M)$ be a group monomorphism.
Then there is a group automorphism $\Phi : H(M) \to H(M)$, and a field homomorphism
$\theta : K \to M$ such that $\Psi = \Phi \circ H(\theta)$.
\end{theorem}
\begin{proof}
Let $M'$ be the algebraic closure of $M$. We consider the following centralizer (in $H(M')$)
$$C:=C_{H(M')}(\Psi(Z(H(K))))$$
of the image by $\Psi$ of the center of $H(K)$.
\\
\\
\textbf{Claim}
\\
$C=H(M')$.
\begin{proof}[Proof of Claim]
Assume not and we will reach a contradiction. Let us notice first that $C$ is a Zariski closed subgroup of $H(M')$ (as any centralizer in any algebraic group). Since $C$ is proper in the connected algebraic group $H(M')$ of dimension 3, we obtain that $\dim(C)\leqslant 2$. Let $C^0$ be the connected component of $C$. We will argue first that $C^0$ is commutative. Since $C^0$ is connected, it is enough to show that the Lie algebra $\mathrm{Lie}(C^0)$ is commutative. However, $\mathrm{Lie}(C^0)$ is nilpotent and of dimension at most two, therefore (by e.g. \cite[Section 3.1]{erdmann_lie}) $\mathrm{Lie}(C^0)$ is commutative.

Since $[C:C^0]$ is finite, $C^0$ is commutative and $H(K)$ embeds into $C$, we obtain that $H(K)$ has a commutative subgroup of finite index. It should be well-known that such subgroups do not exist for infinite fields $K$. To see that, one can e.g. take the Zariski closure in the algebraic group $H(K')$, where $K'$ is the algebraic closure of $K$, and obtain a Zariski closed commutative subgroup of $H(K')$ of finite index. Since $H(K')$ is connected, it is then commutative, a contradiction.
\end{proof}
Since $C=H(M')$, we obtain that $\Psi(Z(H(K)))\subseteq Z(H(M))$. Using this inclusion, let us introduce the following notation for $x,y,z\in K$:
\begin{IEEEeqnarray*}{rCl}
\Psi((x, 0, 0)_K) &  = & h_M(f_1(x), g_1(x), i_1(x)), \\
\Psi((0, y, 0)_K) & = & h_M(f_2(y), g_2(y), i_2(y)), \\
\Psi((0, 0, z)_K) & = & h_M(0,0,i(z)).
\end{IEEEeqnarray*}
For $x, y \in K$, we consider the following commutator:
\begin{IEEEeqnarray*}{rCl}
 (0,0,i(xy)) &  = & \Psi((0,0,xy)) = \Psi(\left[(x,0,0),(0,y,0)\right]) \\
& = & \left[\Psi((x,0,0)),\Psi((0,y,0))\right] \\
& = & \left[(f_1(x), g_1(x), i_1(x)),(f_2(y), g_2(y), i_2(y))\right] \\
& = & (0,0,f_1(x)g_2(y) - f_2(y)g_1(x)).
\end{IEEEeqnarray*}
Therefore, we have
\begin{equation}
i(xy) = f_1(x)g_2(y) - f_2(y)g_1(x).
\end{equation}
Moreover, since $(x,0,0)$ and $(y,0,0)$ commute, we obtain
\begin{IEEEeqnarray*}{rCl}
(0,0,0) & = &   \Psi((0,0,0)) = \Psi(\left[(x,0,0),(y,0,0)\right] \\
& = & \left[\Psi((x,0,0)),\Psi((y,0,0))\right] \\
& = & \left[(f_1(x), g_1(x), i_1(x)),(f_1(y), g_1(y), i_1(y))\right] \\
& = & (0,0, f_1(x)g_1(y) - f_1(y)g_1(x)).
\end{IEEEeqnarray*}
Hence, we have
\begin{equation}
f_1(x)g_1(y) = f_1(y)g_1(x).
\end{equation}
By a similar computation on the commutator of $(0,x,0)$ and $(0,y,0)$, we obtain
\begin{equation}
f_2(x)g_2(y) = f_2(y)g_2(x).
\end{equation}
Let us define
$$d:= i(1) = f_1(1)g_2(1) - f_2(1)g_1(1).$$
Since $\Psi$ is a monomorphism, we get that $d \neq 0$. Using $(1)$--$(3)$, we obtain
$$i(xy)d  =  i(x)i(y).$$
Let us define
$$\theta:= d^{-1}i.$$
Then, we have the following
$$\theta(1) = d^{-1}i(1) = d^{-1}d = 1,$$
$$\theta(xy) = d^{-2}di(xy) = d^{-2}i(x)i(y) = \theta(x)\theta(y).$$
Therefore, $\theta$ is a field homomorphism.

Moreover, we have
\begin{IEEEeqnarray*}{rCl}
dg_1(x) & = &   (f_1(1)g_2(1) - f_2(1)g_1(1))g_1(x) \\
& = & f_1(1)g_1(x)g_2(1) - f_2(1)g_1(1)g_1(x) \\
& = & g_1(1)f_1(x)g_2(1) - g_1(1)g_1(x)f_2(1) \\
& = & g_1(1)i(x).
\end{IEEEeqnarray*}
Similarly, we obtain
$$df_1(x) =  f_1(1)i(x),\quad df_2(x) =  f_2(1)i(x),\quad dg_2(x) = g_2(1)i(x).$$
Thus, we see that
$$f_1 = f_1(1)\theta,\quad g_1 = g_1(1)\theta,\quad f_2 = f_2(1)\theta,\quad g_2 = g_2(1)\theta.$$
Let $\eta:M\to K$ be a $K$-linear map such that $\eta\circ \theta=\id_K$. We clearly have:
$$i_1 = (i_1 \circ \eta) \circ \theta,\quad i_2 = (i_2 \circ \eta) \circ \theta.$$
Let us finally define
$$\psi_1:= i_1 \circ \eta,\quad \psi_2:= i_2 \circ \eta.$$
We see that
\begin{IEEEeqnarray*}{rCl}
\Psi((x,0,0)(y,0,0)) & = &   (f_1(1)\theta(x + y), g_1(1)\theta(x + y), \psi_1(\theta(x) + \theta(y))) \\
& = & \Psi((x,0,0))\Psi((y,0,0)) \\
& = & (f_1(1)\theta(x), g_1(1)\theta(x), \psi_1(\theta(x)))(f_1(1)\theta(y), g_1(1)\theta(y), \psi_1(\theta(y))) \\
& = & (f_1(1)\theta(x + y), g_1(1)\theta(x + y), \psi_1(\theta(x)) + \psi_1(\theta(y))
+ f_1(1)g_1(1)\theta(x)\theta(y)).
\end{IEEEeqnarray*}
Therefore, for any $a, b \in \theta(K)$ we have
$$\psi_1(a +b) = \psi_1(a) + \psi_1(b) + f_1(1)g_1(1)ab.$$
By a similar computation on $(0,x,0)$ and $(0,y,0)$, we see that
$$\psi_2(a +b) = \psi_2(a) + \psi_2(b) + f_2(1)g_2(1)ab.$$
Let us take
$$\Psi_1, \Psi_2:M\to M$$
extending $\psi_1, \psi_2$ respectively which are given by Lemma \ref{psi_lemma}, and let $\Phi$ be the automorphism associated with the matrix
$A = \begin{pmatrix}
f_1(1) & f_2(1) \\
g_1(1) & g_2(1)
\end{pmatrix}$  and the maps $\Psi_1, \Psi_2$. Therefore, we obtain
$$\det(A) = d\quad \text{and}\quad \Psi = \Phi \circ H(\theta),$$
which we needed to show.
\end{proof}

\subsection{Description}\label{gibbssec}
In this subsection, we provide some additional information on the group of automorphisms of $H(K)$. We need the following.

\begin{theorem}[Levchuk]\label{thmlev}
Let $K$ be a field.
\begin{enumerate}
\item If $\mathrm{char}(K)\neq 2$, then
$$\mathrm{Aut}(H(K))\cong \left(\mathrm{End}(K,+)^2\rtimes \mathrm{GL}_2(K)\right)\rtimes \mathrm{Aut}(K).$$

\item If $\mathrm{char}(K)=2$, then
$$\mathrm{Aut}(H(K))\cong \left(\mathrm{End}(K,+)^2\rtimes \left(\left(K^*\right)^2\rtimes \Zz/2\Zz\right)\right)\rtimes \mathrm{Aut}(K).$$
%where $\mathrm{T}_2(K)$ is the group of triangular $2$ by $2$ matrices with coefficients from $K$.
\end{enumerate}
\end{theorem}

\begin{proof}
The first part is included in \cite[Corollary 5]{levchuk83} and we also provide a more explicit argument in Remark \ref{explicit} below.

The second part should follow from the sentence just below the statement of \cite[Lemma 16]{levchuk83}), however, we also provide a quick argument here. As in the proof of Lemma \ref{psi_lemma}, we see that the corresponding maps $\Psi_1,\Psi_2:K\to K$ need to be additive in the case of characteristic 2. In general, a straightforward computation shows that for additive maps $\Psi_1,\Psi_2$, the map
$$\Psi(x,y,z)=(A(x,y),\det(A)z+\Psi_1(x)+\Psi_2(y)+bcxy)$$
is a an automorphism of $H(K)$ if and only if for all $x,x',y,y'\in K$ we have
$$\det(A)xy'+bc(x+x')(y+y')=bc(xy+x'y')+(ax+by)(cx'+dy'),$$
where $A=\begin{pmatrix}
 a & b \\
 c & d
\end{pmatrix}\in \mathrm{GL}_2(K)$. Since $\det(A)\neq 0$, another easy computation shows that the above holds if and only if $A$ is diagonal or anti-diagonal. Such matrices form a subgroup of $\mathrm{GL}_2(K)$ which is isomorphic to $(K^*)^2\rtimes \Zz/2\Zz$, where $\Zz/2\Zz$ acts on $(K^*)^2$ coordinate-wise.
\end{proof}

\begin{remark}\label{explicit}
Gibbs gave in \cite{gibbs70} another description of automorphisms of groups of rational points of unipotent groups. The class of unipotent groups from \cite{gibbs70} is much larger than the one considered by Levchuk in \cite{levchuk83}, however, only fields are considered in \cite{gibbs70} and there are also some restrictions regarding the characteristic. In our case of the Heisenberg group, the case of characteristic two is excluded, which corresponds to two different cases in the statement of Theorem \ref{thmlev}. We list below some types of automorphisms from \cite{gibbs70} (in the special case of the Heisenberg group) and we use them to give a more direct proof of Theorem \ref{thmlev}(1).
\begin{enumerate}
  \item \emph{Central automorphisms} are of the form
  $$(x,y,z)\mapsto (x,y,z+\Psi_1(x)+\Psi_2(y)),$$
  where each $\Psi_i$ is additive. Such automorphisms correspond to the case of $A=I$.

  \item The \emph{graph automorphism} is given as follows
  $$(x,y,z)\mapsto (y,x,-z+xy).$$
This automorphism corresponds to the case $A=\begin{pmatrix}
 0 & 1 \\
 1 & 0
\end{pmatrix}$
 and $\Psi_1=0=\Psi_2$.

  \item \emph{Diagonal automorphisms} are of the form
  $$(x,y,z)\mapsto (ax,by,z)$$
 for $a,b\in K^*$. They correspond to diagonal matrices $A=\begin{pmatrix}
 a & 0 \\
 0 & b
\end{pmatrix}$ and $\Psi_1=0=\Psi_2$.

  \item Suppose that $\mathrm{char}(K)\neq 2$. Then, \emph{extremal automorphisms} are of the form
  $$(x,y,z)\mapsto \left(x+ay,y,z+\frac{a}{2}y^2\right)$$
for $a\in K$. These automorphisms correspond to the case of $A=\begin{pmatrix}
 1 & a \\
 0 & 1
\end{pmatrix}$,
 $\Psi_1=0$ and $\Psi_2(y)=\frac{a}{2}y^2$.

\end{enumerate}
It is easy to see that the matrices $A$ appearing in Items $(2)$--$(4)$ above generate $\mathrm{GL}_2(K)$ which gives another argument for Theorem \ref{thmlev}(1).
\end{remark}

\section{Model completeness}\label{secmc}
\subsection{Interpretations}
 For precise definitions regarding interpretations, we refer to \cite{HoMo}. A structure $\mathcal{N}$ is \emph{interpretable} in a structure $\mathcal{M}$, if every set definable in $\mathcal{N}$ (including its universe) is also interpretable in  $\mathcal{M}$. Such an interpretation yields the following functor
$$\Gamma:\textbf{Models}(\mathrm{Th}(\mathcal{M}))\to \textbf{Models}(\mathrm{Th}(\mathcal{N})),$$
where for a theory $T$, $\textbf{Models}\left(T\right)$ is the category in which the objects are models of $T$ and the morphisms are elementary embeddings between them. For more details, see \cite[Theorem 5.3.3]{HoMo}.

Let $K$ be a field. We have the obvious interpretation of the field $H(K)$ in the group $K$ and the functor induced by this interpretation is $H$. We consider the language $\mathcal{L}$, which is the language of groups with two additional constant symbols. We recall Maltsev's interpretation of the field $K$ in $H(K)$ (regarded as an $\mathcal{L}$-structure) following  \cite{Int} and \cite{danmya} below, since we need its specific form for an application later (see Remark \ref{fieldint2}(2) and Remark \ref{last}).
\begin{fact}[Maltsev \cite{Maltsev}]\label{fieldint}
The field $K$ is interpretable in the $\mathcal{L}$-structure
$$\left(H(K);\cdot,(1,0,0)_K,(0,1,0)_K\right)$$
in the following way.
\begin{enumerate}
  \item The universe is $Z(H(K))$ which is a definable subset of $H(K)$.

  \item The formula defining the field addition is just the group multiplication
$$\oplus(x,y,z) := xy = z.$$

  \item The formula for the graph of the field multiplication $\otimes(x, y, z)$ is given below:
$$(\exists x' \exists y')\left([x',u]=[y',v]=I \ \wedge \ [x',v] = x \ \wedge \ [u,y'] = y \ \wedge \ [x',y']= z\right),$$
where $I=(0,0,0)$ denotes the identity element of the Heisenberg group and
$$u:=(1,0,0),\quad v:=(0,1,0).$$
\end{enumerate}
Moreover, this is an ``$\exists_1^+$-interpretation'' in the terminology from \cite{HoMo}, that is it is given by positive existential formulas.
\end{fact}
\begin{remark}\label{fieldint2}
We comment here on the interpretation from Lemma \ref{fieldint}.
\begin{enumerate}
  \item We get the following interpretation functor
 $$\Theta : \textbf{Models}\left(\Th\left(H(K); \cdot,(1,0,0),(0,1,0)\right)\right) \to \textbf{Models}\left(\Th(K) \right).$$

  \item By the moreover claim from Lemma \ref{fieldint} and \cite[Theorem 5.3.4(b)]{HoMo}, the functor from Item $(1)$ extends to a functor
 $$\Theta_{1-1} : \textbf{Models}_{1-1}\left(\Th\left(H(K); \cdot,(1,0,0),(0,1,0)\right)\right) \to \textbf{Models}_{1-1}\left(\Th(K) \right),$$
 where for any theory $T$, we denote by $\textbf{Models}_{1-1}(T)$ the category of models of $T$ with embeddings (\cite[Theorem 5.3.4(b)]{HoMo} even allows to replace embeddings with arbitrary homomorphisms, but in our case they are embeddings anyway).

  \item The map
$$t_K:(K;+,\cdot)\to (Z(H(K));\oplus,\otimes),\quad t_K(x)=(0,0,x)_K$$
is a natural and $K$-definable isomorphism of fields, that is for any field homomorphism $\alpha:K\to M$ the following diagram is commutative.
\begin{center}
\begin{tikzcd}
\Theta(H(K)) \arrow[r, "\Theta(H(\alpha))"]  & \Theta(H(M))
\\
K \arrow[r, "\alpha"] \arrow[u, "t_K"] & M \arrow[u, "t_M"]
\end{tikzcd}
\end{center}
Thus, we have an isomorphism (given by a definable map) between the identity functor on ${\normalfont \textbf{Models}}_{1-1}(\Th(K))$ and the composition functor $\Theta_{1-1}\circ H$.
\end{enumerate}
\end{remark}

\subsection{Test for model completeness and main result}

We follow here briefly the presentation from \cite{HKTY}.
\begin{definition}\label{mcparam}
Let $\mathcal{L}$ be a language and $M$ be an $\mathcal{L}$-structure. We say that $M$ is \emph{model complete} if $\mathrm{Th}(M)$ is
    model complete.
\end{definition}
\begin{remark}\label{mcspecial}
To test model completeness of a theory $T$, it is enough to consider embeddings between \emph{special models}  (as in \cite[Section 10.4]{HoMo}) of $T$. \end{remark}
We will need the following result which was suggested by Will Johnson.

\begin{theorem}[Theorem 2.17 in \cite{HKTY}]\label{pkthm}
Suppose
$$\Gamma :{\normalfont \textbf{Models}}(T_1) \to {\normalfont \textbf{Models}}(T_2)$$
is an interpretability functor and  $M_2\models T_2$ is special.
If there is $M_1\models T_1$ such that $M_2\equiv \Gamma(M_1)$, then there is $M_1'\models T_1$ such that $M_2\cong \Gamma(M_1')$.
\end{theorem}

We are ready now to show the main result of this paper.

\begin{theorem}
Let $K$ be a field. Then the group $H(K)$ is model complete if and only if $K$ is a model complete field.
\end{theorem}
\begin{proof}
Assume that $K$ is model complete. Let $G \equiv H(K) \equiv N$ and $f : G \to N$ be a group monomorphism. We need to show that $f$ is elementary. By Remark \ref{mcspecial}, we can assume that $H$ and $N$ are special. By applying Theorem \ref{pkthm} to the interpretability functor
$$H :
{\normalfont \textbf{Models}}\left(\Th(K)\right) \to
{\normalfont \textbf{Models}}\left(\Th(H(K))\right),$$
there are fields $F$, $M$ such that
$$G \cong H(F),\quad N \cong H(M),\quad F \equiv K \equiv M.$$
Since any isomorphism is elementary and the composition of elementary maps is elementary, we can also assume that $f : H(F) \to H(M)$. By Theorem \ref{homthm}, we can finally assume that $f = H(\alpha)$ for some field monomorphism $\alpha: F \to M$. Since $K$ is model complete and $F \equiv K \equiv M$, we get that the map $\alpha$ is elementary. Because interpretability functors take elementary embeddings to elementary embeddings, we conclude that $f$ is elementary.

We assume now that the group $H(K)$ is model complete and let
$$T:= \Th\left((H(K), \cdot,(1,0,0),(0,1,0)) \right).$$
The $\mathcal{L}$-theory $T$ is still model complete, since we only added constant symbols to the language. Let $F$ and $M$ be fields such that
$$F \equiv K \equiv M$$
and $\alpha : F \to M$ a field monomorphism. Then, we have
$$H(F) \equiv H(K) \equiv H(M).$$
By our assumption, the map
$$H(\alpha): H(F) \to H(M)$$
 is elementary. By interpreting the extra constant symbols from the language $\mathcal{L}$ as $(1,0,0)$ and $(0,1,0)$ in both structures, we can treat $H(F)$ and $H(M)$ as models of $T$.
Moreover, $H(\alpha)$ sends $(1,0,0)_F$ and $(0,1,0)_F$ to $(1,0,0)_M$ and $(0,1,0)_M$ respectively, thus it is also an elementary embedding of models of $T$. By Remark \ref{fieldint2}(1), we have that $\beta := \Theta(H(\alpha))$ is elementary. By Remark \ref{fieldint2}(2), we get that $\alpha = \Psi \circ \beta \circ \Phi$ for some isomorphisms $\Psi, \Phi$. Therefore, we conclude that $\alpha$ is elementary.
\end{proof}

\begin{remark}\label{last}
We comment here on some natural generalizations of the arguments from the last proof.
\begin{enumerate}
  \item Let us assume we are in the set-up from Remark \ref{fieldint2}, that is we have an interpretability functor $\Gamma$ which extends to a functor
   $$\Gamma : \textbf{Models}_{1-1}\left(T\right) \to \textbf{Models}_{1-1}\left(T'\right)$$
having a right quasi-inverse as in Remark \ref{fieldint2}(3). In such a case, if $T$ is model complete then $T'$ is model complete as well, which was observed in the case of existential bi-interpretations in \cite[Corollary 2.16]{intfus1}.

\item There is a natural generalization of the first implication from the above proof as well. Let us assume we have an interpretability functor $\Gamma$ which extends to a functor
   $$\Gamma : \textbf{Models}_{1-1}\left(T\right) \to \textbf{Models}_{1-1}\left(T'\right)$$
   satisfying a ``weak Borel-Tits property'', that is for all $\mathcal{M}_1,\mathcal{M}_2\models T$ and for all
   $$f:\Gamma(\mathcal{M}_1)\hookrightarrow \Gamma(\mathcal{M}_2)$$
   there are $\alpha:\mathcal{M}_1\hookrightarrow \mathcal{M}_2$ and $t\in \mathrm{Aut}(\mathcal{M}_2)$ such that
   $$f=t\circ \Gamma(\alpha).$$
   In such a case, if $T$ is model complete then $T'$ is model complete as well.

This property was used to show model completeness of some groups in \cite{HKTY}, and it was applied there in the form of the actual Borel-Tits theorem (see \cite[(A)]{Botits} and \cite[Theorem 1.3]{Steinberg}).

\end{enumerate}
\end{remark}

\section{Other model-theoretic properties of $H(K)$}\label{seclast}

In this section, we focus on other model-theoretic properties of the Heisenberg group and our results are mostly negative. We consider first the question of quantifier elimination, which can be settled easily using the work of Cherlin and Felgner from \cite{cherlin82}.
\begin{theorem}
If $K$ is an infinite field, then the group $H(K)$ does not have quantifier elimination.
\end{theorem}
\begin{proof}
If $\mathrm{char}(K)=0$, then $H(K)$ is torsion-free. Therefore, we can use e.g. \cite[Theorem 3.4]{cherlin82} which says that locally solvable torsion-free groups with quantifier elimination are necessarily commutative (and divisible).

If $\mathrm{char}(K)=p>0$, then $H(K)$ is a $p$-group, that is the order of each element of $H(K)$ is a power of $p$ (here, at most $p^2$). Then, we can use e.g. \cite[Theorem 4.2]{cherlin82} which says that hipercentral (in particular, nilpotent) $p$-groups with quantifier elimination are commutative.
\end{proof}

There are several questions in Section 5 of \cite{Int} concerning the model theory of Heisenberg groups. One of them regards the (effective) bi-interpretability of $K$ and $H(K)$. Regarding this question, we point out the following negative result.
\begin{prop}\label{notbiint}
If $K$ is an infinite field, then the interpretation $K\mapsto H(K)$ cannot be ``inverted to a bi-interpretation'' even using extra parameters.
\end{prop}
\begin{proof}
Let us suppose there is an interpretation of $K$ in $H(K)$ ``inverting'' the interpretation $K\mapsto H(K)$ for an uncountable field $K$ using parameters contained in a countable subfield $K_0\subset K$. Then, by e.g. \cite[Section 5.4, Exercise 8(b)]{HoMo}, we get the induced isomorphism:
$$H_K:\mathrm{Aut}_{\mathrm{fields}}(K/K_0)\to \mathrm{Aut}_{\mathrm{groups}}(H(K)/H(K_0)).$$
However, there many automorphisms of $H(K)$ fixing pointwise any countable subgroup $H_0<H(K)$ which are not coming from the field automorphisms, e.g. the central automorphisms from Remark \ref{explicit}(1).
\end{proof}
\begin{remark}
One can also argue above by showing that the interpretation of $H(K)$ in $K$ is not \emph{full} (see \cite[Def. 2.1(3)]{CasHas}), since, for example,
the projection map
$$p:H(K)\to Z(H(K)),\quad p((a,b,c)_K)=(0,0,c)_K$$
is not definable in the pure group $H(K)$ (even using parameters) which can be seen by an automorphism argument similarly as in the proof of Proposition \ref{notbiint}.
\end{remark}
Regarding the ``real'' question about \emph{any} possible bi-interpretability between $K$ and $H(K)$, the situation is more complicated. It is suggested in \cite{Int} to show that the groups $\mathrm{Aut}(K)$ and $\mathrm{Aut}(H(K))$ are not isomorphic and this is the route we take.  We could success in some particular cases only and these cases are listed below.
\begin{theorem}\label{notbint}
A field $K$ is not bi-interpretable with the group $H(K)$, if $K$ comes from the following list.
\begin{enumerate}
  \item Finitely generated fields.

  \item Real closed fields.

\item $p$-adically closed fields.

  \item Algebraically closed fields.
\end{enumerate}
\end{theorem}
\begin{proof}
%We use the following isomorphism
%$$\mathrm{Aut}(H(K))\cong \left(\mathrm{End}(K,+)^2\rtimes \mathrm{GL}_2(K)\right)\rtimes \mathrm{Aut}(K).$$
By Theorem \ref{thmlev}, we get the following
$$|\mathrm{Aut}(H(K))|\geqslant |\mathrm{End}(K,+)|$$
for any field $K$. Since for finitely generated fields we have
$$|\mathrm{End}(K,+)|>|\mathrm{Aut}(K)|,$$
Item $(1)$ follows.

Assume that $K$ is bi-interpretable with $H(K)$ for a real closed field $K$. Then, we also get that $\Rr$ is bi-interpretable with $H(\Rr)$, since the theory of real closed fields is complete. Thus, we can conclude as above, since $\mathrm{Aut}(\Rr)$ is trivial. The argument for $p$-adically closed fields is analogous using that the group $\mathrm{Aut}(\Qq_p)$ is trivial as well.

Regarding the last case, assume that $K$ is an algebraically closed field and let $F$ be the prime subfield of $K$. We have the following exact sequence of groups
$$1\longrightarrow \mathrm{Aut}\left(K/F^{\alg}\right)\longrightarrow \mathrm{Aut}\left(K\right)
\longrightarrow \mathrm{Gal}(F)\longrightarrow 1,$$
where $\mathrm{Gal}(F)$ is the absolute Galois group of $F$, so $|\mathrm{Gal}(F)|=2^{\aleph_0}$. Lascar showed that the group $\mathrm{Aut}\left(K/F^{\alg}\right)$ is simple (it was proved for $K=\Cc$ in \cite{lascar_simple} and the general case was shown in \cite{BCHMP}).
In particular, if $H$ is a normal subgroup of $\mathrm{Aut}(K)$, then we have
$$|H|\leqslant 2^{\aleph_0}\quad \text{or} \quad |\mathrm{Aut}(K):H|\leqslant 2^{\aleph_0}.$$
Assume now that $K$ is bi-interpretable with $H(K)$ and we will reach a contradiction. Without loss of generality, $K$ is uncountable. By Theorem \ref{thmlev}, there is $N\leqslant \mathrm{GL}_2(K)$ such that
$$\mathrm{Aut}(K)\cong \mathrm{Aut}(H(K))\cong \left(\mathrm{End}(K,+)^2\rtimes N\right)\rtimes \mathrm{Aut}(K).$$
Let us take the normal subgroup $H$ of $\mathrm{Aut}(K)$ corresponding to $\mathrm{End}(K,+)^2$ under the isomorphism above.
We obtain that
$$|H|> 2^{\aleph_0}\quad \text{and} \quad |\mathrm{Aut}(K):H|> 2^{\aleph_0},$$
which gives a contradiction.
\end{proof}

\begin{remark}
If $K$ is an infinite field coming from the list given in the statement of Theorem \ref{notbint}, then a similar proof also shows that $K$ and $H(K)$ are not bi-interpretable even using extra parameters. We have tried to show the same for any infinite field, but we could not. The existence of a field $K$ such that for all $M\equiv K$ we have a topological isomorphism
$$\mathrm{Aut}(M)\cong \left(\mathrm{End}(M,+)^2\rtimes \mathrm{GL}_2(M)\right)\rtimes \mathrm{Aut}(M)$$
(see Section \ref{gibbssec}) should lead to a contradiction, but we were unable to obtain it.

It is shown in \cite{khelif} and \cite{niesgps} that $\Zz$ and $H(\Zz)$ are not bi-interpretable (even using extra parameters) by finding a ``large'' (here: recursively saturated) ring $R\equiv \Zz$ such that
$$|\mathrm{Aut}(H(R))|>|\mathrm{Aut}(R)|.$$
A corresponding property does not hold for arbitrary fields (in place of the ring $\Zz$). For example, it does not hold for algebraically closed fields $K$, since for such fields we always have
$$|\mathrm{Aut}(H(K))|=|\mathrm{Aut}(K)|.$$
\end{remark}

\bibliographystyle{plain}
\bibliography{harvard}

\end{document}